\documentclass[reqno,12pt]{amsproc}

\usepackage{amsmath,amssymb,amsthm,embedfile,fullpage,hyperref}

\DeclareMathOperator{\dom}{dom}
\DeclareMathOperator{\SL}{SL}

\newcommand{\cE}{\mathcal{E}}
\newcommand{\cJ}{\mathcal{J}}
\newcommand{\cS}{\mathcal{S}}
\newcommand{\dC}{\mathbb{C}}
\newcommand{\dR}{\mathbb{R}}
\newcommand{\dZ}{\mathbb{Z}}

\newcommand{\gp}{\mathsf{GP}}
\newcommand{\hlc}{\mathsf{HLC}}
\newcommand{\pp}{\mathsf{PP}}
\newcommand{\SC}{\mathsf{SC}}
\newcommand{\sg}{\mathsf{SG}}
\newcommand{\sq}{\mathsf{SQ}}
\newcommand{\tg}{\mathsf{TG}}
\newcommand{\tgp}{\mathsf{TGP}}

\newcommand{\smat}[4]{\left(\begin{smallmatrix} #1 & #2 \\ #3 & #4 \end{smallmatrix}\right)}
\newcommand{\tH}{\tilde H}

\newtheorem{definition}{Definition}
\newtheorem{theorem}[definition]{Theorem}
\newtheorem{proposition}[definition]{Proposition}
\theoremstyle{definition}
\newtheorem{example}{Example}
\numberwithin{equation}{section}
\numberwithin{definition}{section}

\begin{document}
\title{``Graph Paper'' Trace Characterizations of Functions of Finite Energy}
\author{Robert S. Strichartz}
\thanks{Research supported in part by the National Science Foundation, grant DMS-1162045.}
\address{Mathematics Department, Malott Hall, Cornell University, Ithaca, NY 14853}
\email{str@math.cornell.edu}
\date{}

\begin{abstract}
We characterize functions of finite energy in the plane in terms of their 
traces on the lines that make up ``graph paper'' with squares of side length 
$m^n$ for all $n$, and certain $\frac 1 2$-order Sobolev norms on the graph 
paper lines. We also obtain analogous results for functions of finite energy 
on two classical fractals: the Sierpinski gasket and the Sierpinski carpet. 
\end{abstract}

\maketitle

\section{Introduction}\label{sec:!}

Functions of finite energy play an important role in analysis and probability. 
On Euclidean space or a domain in Euclidean space, these are just the 
functions whose gradient in the distribution sense belongs to $L^2$, with 
the energy given by 
\begin{equation}\label{eq:1.1}
  \int \left| \nabla F\right|^2\, dx
\end{equation}
As such they make up a homogeneous Sobolev space that we will denote here as 
$H^1$. The more usual inhomogeneous Sobolev space is smaller, requiring that 
$F\in L^2$ as well \cite{9,10}. There are many ways to generalize the notion 
of finite energy to other contexts. For example, as the functions in the 
domain of a Dirichlet form \cite{4}. In this paper we will only consider 
functions of finite energy in regions in the plane, and on two classical 
fractals, the Sierpinski gasket \cite{8,12} and the Sierpinski carpet 
\cite{1,2}. 

It is well-known that functions of finite energy in the plane (or in higher 
dimensions) do not have to be continuous, so the value $F(x,y)$ at a point is 
not meaningful. Nevertheless, the trace on a line, say $TF(x) = F(x,0)$, is 
well-defined, and belongs to a certain $\frac 1 2$-order homogeneous Sobolev 
space that we will denote here by $H^{1/2}(\dR)$, defined by the finiteness of 
\begin{equation}\label{eq:1.2}
  \int_{--\infty}^\infty \int_{-\infty}^\infty \frac{\left|f(x)-f(y)\right|^2}{|x-y|^2}\, dxdy\text{,}
\end{equation}
with a corresponding norm estimate. Of course it is the norm estimate that is 
important since it implies the existence of the trace by routine arguments. 
The result is sharp, meaning that there is an extension operator from 
$H^{1/2}(\dR)$ to $H^1(\dR^2)$. There are in fact two rather natural 
$\frac 1 2$-order Sobolev spaces on $\dR$. The other one, which we denote by 
$\tH^{1/2}(\dR)$ is larger, and only requires the finiteness of an integral like 
(\ref{eq:1.2}) where the integration is restricted to the region 
$|x-y|\leqslant 1$. We will show that the trace of a function $F$ of finite 
energy in the strip $\{(x,y) : 0<y<1\}$ only belongs to $\tH(\dR)$. In 
particular this implies that there does not exist a Sobolev extension theorem 
from $H^1$ of the strip to $H^1(\dR^2)$, even though such a result for 
inhomogeneous Sobolev spaces is well-known and essentially trivial. 

The trace of a function of finite energy on a single line does not, of course, 
determine the function. What about the trace of an infinite collection of 
lines that together form a dense subset of the plane? A simple example is the 
set of lines of ``graph paper,'' where we take the graph paper squares to have 
side length $m^n$, where $m$ is an integer ($m\geqslant 2$) and $n$ varies 
over $\dZ$, so the graph papers $\gp_{m^n}$ are nested. The main results of 
this paper are first a trace theorem that characterizes the traces of 
$H^1(\dR^2)$ functions on $\gp_{m^n}$ in terms of a Sobolev space 
$H^{1/2}(\gp_{m^n})$ with a given norm, and then the characterization of 
$H^1(\dR^2)$ in terms of a uniform bound on the norms of the traces on 
$\gp_{m^n}$ as $n\to -\infty$. 
The trace theorem is discussed in section \ref{sec:3} in the context of 
Sobolev spaces $H^{1/2}$ on \emph{metric graphs} (graphs whose edges have 
specified length, \cite{3}), as discussed in section \ref{sec:2}. Because the 
functions in these spaces need not be continuous, the key issue is to 
understand a kind of ``gluing'' condition at the vertices of the graph. It 
turns out that this condition was given in \cite{11}. For the 
convenience of the reader we give all the proofs in section \ref{sec:2}, 
although many of the results are already known, because they are usually 
treated in the context of inhomogeneous Sobolev spaces. In section \ref{sec:4} 
we discuss the trace characterizations of $H^1(\dR^2)$. In section \ref{sec:5} 
we discuss the analogous results on the two fractals. It turns out that the 
trace theorems are already known \cite{5,6,7}, and the Sobolev spaces are 
$H^\beta$ for values satisfying $\frac 1 2<\beta<1$. The spaces of functions 
of finite energy on these fractals consist of continuous functions, as do the 
trace spaces, so there is no difficulty defining the traces, and the 
``gluing'' condition at vertices is simply continuity. Thus the fractal analog 
of the trace characterization is perhaps simpler than the theorem in the 
plane. We also characterize the traces on Julia sets of functions of finite 
energy in the unbounded component of the complement of the Julia set. We 
believe strongly that there is a great benefit to thinking about problems in 
both the smooth and the fractal contexts, and looking for interactions in the 
ideas that emerge. We hope this paper gives some support to this point of 
view.

\section{Metric Graphs}\label{sec:2}

A \emph{metric graph} $G=(V,E,L_e)$ consists of a graph $(V,E)$ with vertices 
$V$ and edges $E$, and a function that assigns a length $L_e$ in $(0,\infty]$ 
to each edge $e\in E$. 

\begin{definition}\label{dfn:2.1}
For a metric graph $G=(V,E,L_e)$, define the homogeneous Sobolev norm 
\begin{equation}\label{eq:2.1}
\begin{aligned}
  \left\|f\right\|_{H^{1/2}(G)}^2 
    &= \sum_{e\in E} \int_0^{L_e} \int_0^{L_e} \frac{\left|f\left(e(x)\right) - f\left(e(y)\right)\right|^2}{|x-y|^2}\, dxdy \\
    &+ \sum_{e\sim e'} \int_0^L \frac{\left|f\left(e(x)\right) - f\left(e'(x)\right)\right|^2}{x}\, dx
\end{aligned}
\end{equation}
(in the second sum $L=\min(L_e,L_{e'})$, and the parameterizations of $e$ and 
$e'$ are chosen so that $e(0)$ and $e'(0)$ correspond to the intersection 
point). We define the Sobolev space $H^{1/2}(G)$ to be the equivalence classes 
(modulo constants) of locally $L^2$ functions for which the norm is finite. It 
is easy to see that $H^{1/2}(G)$ is a Hilbert space. 
\end{definition}

\begin{example}\label{ex:1}
Let $G=\dR$, so $G$ has no vertices and a single edge of infinite length. We 
need to modify (\ref{eq:2.1}) in this case to read 
\begin{equation}\label{eq:2.2}
  \left\|f\right\|_{H^{1/2}(\dR)}^2 = \int_{-\infty}^\infty \int_{-\infty}^\infty \frac{\left|f(x)-f(y)\right|^2}{|x-y|^2}\, dxdy\text{.}
\end{equation}
For this example we also want to consider the smaller norm 
\begin{equation}\label{eq:2.3}
  \left\|f\right\|_{\tH^{1/2}(\dR)}^2 = \iint_{|x-y|\leqslant 1} \frac{\left|f(x)-f(y)\right|^2}{|x-y|^2} \, dxdy
\end{equation}
and corresponding larger Sobolev space $\tH^{1/2}(\dR)$. 
\end{example}

We note that the space $H^{1/2}(\dR)$ is M\"obius invariant, meaning that 
$f\in H^{1/2}(\dR)$ if and only if $f\circ M\in H^{1/2}(\dR)$ with equal 
norms, for $M(x)=\frac{ax+b}{cx+d}$ with $\smat a b c d\in \SL(2,\dR)$. 
Indeed it suffices to verify this for translations $M(x)=x+b$, dilations 
$M(x) = a x$ and the inversion $M(x)=\frac 1 x$, where it follows by a 
simple change of variable in the integral defining the norm. We note that the 
same statement is false for $\tH^{1/2}(\dR)$. 

We may easily characterize these norms and spaces in terms of the Fourier 
transform $\hat f$. The finiteness of the norm easily implies that $f$ is a 
tempered distribution so $\hat f$ is well defined as a tempered distribution, 
and the equivalence of the functions that differ by a constant means $\hat f$ 
is only defined up to the addition of an arbitrary multiple of the delta 
function. Note that there is no ``canonical'' choice of $f$ and $\hat f$ 
within each equivalence class. 

\begin{theorem}\label{thm:2.2a}
a) $f\in H^{1/2}(\dR)$ if and only if $\hat f$ may be identified with a 
function that is locally in $L^2$ in the complement of the origin with 
\begin{equation}\label{eq:2.4}
  \int_{-\infty}^\infty |\hat f(\xi)|^2 |\xi|\, d\xi < \infty\text{,}
\end{equation}
and (\ref{eq:2.4}) is in fact a constant multiple of (\ref{eq:2.2})

b) $f\in \tH^{1/2}(\dR)$ if and only if $\hat f$ may be identified with a 
function that is locally in $L^2$ in the compliment of the origin, with 
\begin{equation}\label{eq:2.5}
  \int_{|\xi|\geqslant 1} |\hat f(\xi)|^2|\xi|\, d\xi + \int_{|\xi|\leqslant 1} |\hat f(\xi)|^2|\xi|^2\, d\xi < \infty\text{,}
\end{equation}
and (\ref{eq:2.5}) is bounded above and below by a multiple of (\ref{eq:2.3}). 
\end{theorem}
\begin{proof}
a) is of course well-known, and follows from the formal computation 
\begin{align*}
  \|f\|_{H^{1/2}(\dR)}^2 
    &= \int_{-\infty}^\infty \int_{-\infty}^\infty |f(x+t)-f(x)^2\, dx\frac{dt}{t^2} \\
    &= \int_{-\infty}^\infty \int_{-\infty}^\infty |\hat f(\xi)|^2 |e^{2\pi i \xi t}-1|^2\, d\xi\frac{dt}{t^2} \\
    &= c \int_{-\infty}^\infty |\hat f(\xi)|^2|\xi|\, d\xi
\end{align*}
for $c=\int_{-\infty}^\infty \frac{|e^{2\pi i t}-1|^2}{t^2}\, dt$. 

To prove b) we similarly compute 
\begin{align*}
  \|f\|_{\tH^{1/2}(\dR)}^2 
    &= \int_{-1}^1 \int_{-\infty}^\infty |f(x+t)-f(x)^2\, dx\frac{dt}{t^2} \\
    &= \int_{-\infty}^\infty |\hat f(\xi)|^2 \left(\int_{-1}^1 |e^{2\pi i \xi t}-1|^2\, \frac{dt}{t^2}\right)\, d\xi
\end{align*}
Now 
\[
  \int_{-1}^1 |e^{2\pi i \xi t} - 1|^2\, \frac{dt}{t^2} = |\xi| \int_{-|\xi|}^{|\xi|} |e^{2\pi i t}-1|^2\, \frac{dt}{t^2}\text{,}
\]
and for $|\xi|\geqslant 1$ the last integral is bounded above and below by a 
constant. On the other hand, for $|\xi|\leqslant 1$, the integrand is bounded 
above and below by a constant, so the integral is bounded above and below by 
the length of the interval. This shows the equivalence of (\ref{eq:2.3}) and 
(\ref{eq:2.5}).

The formal computation easily implies that any $f\in \tH^{1/2}(\dR)$ has a 
Fourier transform satisfying (\ref{eq:2.5}). To complete the proof we need to 
show that any locally $L^2$ function $g(\xi)$ with 
\begin{equation}\label{eq:2.6}
  \int_{|\xi|\geqslant 1} |g(\xi)|^2|\xi|\, d\xi + \int_{|\xi|\leqslant 1} |g(\xi)|^2 |\xi|^2\, d\xi < \infty
\end{equation}
is in fact the Fourier transform of a function in $\tH^{1/2}(\dR)$. Since the 
only problem is near the origin, we may assume that $g$ is supported in 
$[-1,1]$. Let $h(\xi) = \xi g(\xi)$. Note that $h\in L^2$ by (\ref{eq:2.6}). 
We define a distribution $\tilde g$ associated to $g$ as follows. Note that 
$\langle \tilde g,\varphi\rangle = \int h(\xi)\varphi(\xi)\,\frac{d\xi}{\xi}$ 
is well-defined for any $\varphi\in \cS$ with $\varphi(0) = 0$. Choose 
$\psi\in \cS$ with $\psi(0) = 1$. Then 
$\varphi(\xi) = \left(\varphi(\xi)-\varphi(0)\psi(\xi)\right)+\varphi(0)\psi(\xi)$, 
with the first summand vanishing at the origin. We will choose to have 
$\langle \tilde g,\psi\rangle = 0$, so our definition of $\tilde g$ is 
\begin{equation}\label{eq:2.7}
  \langle \tilde g,\varphi\rangle = \int h(\xi) \left(\varphi(\xi)-\varphi(0)\psi(\xi)\right)\,\frac{d\xi}{\xi}\text{.}
\end{equation}
It follows that $h(\xi)=\xi \tilde g(\xi)$ in the distribution sense. The 
inverse Fourier transform of $\tilde g$ is the function $f$. Note that $f$ has 
a derivative in $L^2$ so it is continuous, and the formal computation shows 
that $f\in \tH^{1/2}(\dR)$. 
\end{proof}

A trivial consequence of the theorem is that the space $\tH^{1/2}(\dR)$ is 
strictly larger than $H^{1/2}(\dR)$. On the other hand, 
$L^2(\dR)\cap \tH^{1/2}(\dR) = L^2(\dR)\cap H^{1/2}(\dR)$. 

\begin{example}\label{ex:2}
Let $G$ be the graph with one vertex and two edges of infinite length meeting 
at the vertex. We may realize $G$ as the real line with edges $(-\infty,0]$ 
and $[0,\infty)$, and we write it as $\dR^-\cup \dR^+$. We see that 
(\ref{eq:2.1}) explicitly is 
\begin{equation}\label{eq:2.8}
\begin{aligned}
  \|f\|_{H^{1/2}(\dR^-\cup \dR^+)}^2 
    &= \int_{-\infty}^0 \int_{-\infty}^0 \frac{|f(x)-f(y)|}{|x-y|}\, dxdy \\
    &+ \int_0^\infty \int_0^\infty \frac{|f(x)-f(y)|^2}{|x-y|^2}\, dxdy\\
    &+ \int_0^\infty \frac{|f(x)-f(-x)|^2}{x}\, dx
\end{aligned}
\end{equation}
\end{example}

\begin{theorem}\label{thm:2.2b}
The spaces $H^{1/2}(\dR^-\cup \dR^+)$ and $H^{1/2}(\dR)$ are identical 
with equivalent norms. 
\end{theorem}
\begin{proof}
This result is essentially contained in \cite{11}, section III.3. Let $x,y$ 
stand for variables that are always positive. Since 
$\int_0^\infty \frac{dy}{(x+y)^2} = \frac 1 x$ we have 
\[
  \int_0^\infty \frac{|f(x)-f(-x)|}{x}\, dx = \int_0^\infty \int_0^\infty \frac{|f(x)-f(-x)|^2}{(x+y)^2}\, dydx\text{.}
\]
Writing $f(x)-f(-x) = \left(f(x)-f(y)\right)+\left(f(y)-f(-x)\right)$, 
we have by the triangle inequality 
\begin{align*}
  \left(\int_0^\infty \frac{|f(x)-f(-x)|^2}{x}\, dx\right)^{1/2} 
    &\leqslant \left(\int_0^\infty \int_0^\infty \frac{|f(x)-f(y)|^2}{(x+y)^2}\, dydx\right)^{1/2} \\
    &+ \left(\int_0^\infty \int_0^\infty \frac{|f(y)-f(-x)|^2}{|x+y|^2}\, dydx\right)^{1/2} \\
    &\leqslant 2\|f\|_{H^{1/2}(\dR)}
\end{align*}
since $\frac{1}{(x+y)^2}\leqslant \frac{1}{(x-y)^2}$. This yields the bound of 
(\ref{eq:2.8}) by a multiple of $\|f\|_{H^{1/2}(\dR)}$. A similar argument 
gives 
\begin{align*}
  \left(\int_0^\infty\int_0^\infty\frac{|f(y)-f(-x)|^2}{|x+y|^2}\ dydx\right)^{1/2} 
    &\leqslant \left(\int_0^\infty \int_0^\infty \frac{f(x)-f(y)|^2}{|x+y|^2}\, dydx\right)^{1/2} \\
    &+ \left(\int_0^\infty \frac{|f(x)-f(-x)|^2}{x}\, dx\right)^{1/2}
\end{align*}
for the bound in the other direction. 
\end{proof}

\begin{example}\label{ex:3}
Let $G$ be the graph $\dZ$; in other words the vertices are the integers and 
the edges are $[k,k+1]$ for $k\in\dZ$ of length $1$. Then (\ref{eq:2.1}) is 
explicitly 
\begin{equation}\label{eq:2.9}
  \|f\|_{H^{1/2}(\dZ)}^2 = \sum_{k\in\dZ} \int_k^{k+1}\int_k^{k+1} \frac{|f(x)-f(y)|^2}{|x-y|^2} \, dxdy + \sum_{k\in\dZ} \int_0^1 \frac{|f(k+t)-f(k-t)|}{t}\, dt\text{.}
\end{equation}
\end{example}

\begin{theorem}\label{thm:2.3}
The spaces $H^{1/2}(\dZ)$ and $\tH^{1/2}(\dR)$ are identical with equivalent 
norms. 
\end{theorem}
\begin{proof}
The first term on the right side of (\ref{eq:2.9}) is clearly bounded by 
$\|f\|_{\tH^{1/2}(\dR)}^2$. For the second term we note that an argument as in 
the proof of Theorem \ref{thm:2.2b} gives the estimate 
\[
  \int_0^1 \frac{|f(k+t)-f(k-t)|^2}{t}\, dt\leqslant c\int_{k-1}^{k+1}\int_{k-1}^{k+1} \frac{|f(x)-f(y)|^2}{|x-y|^2}\, dxdy\text{,}
\]
and summing over $k\in\dZ$ we obtain 
\[
  \sum_{k\in\dZ} \int_0^1 \frac{|f(k+t)-f(k-t)|^2}{t}\, dt \leqslant c \iint_{|x-y|\leqslant 2} \frac{|f(x)-f(y)|^2}{|x-y|^2}\, dxdy\text{.}
\]
A straightforward estimate controls the integral over 
$1\leqslant |x-y|\leqslant 2$ by a multiple of the integral over 
$|x-y|\leqslant 1$, so we have 
\[
  \|f\|_{H^{1/2}(\dZ)}^2\leqslant c \|f\|_{\tH^{1/2}(\dR)}^2\text{.}
\]
For the reverse estimate we use an argument in the proof of Theorem 
\ref{thm:2.2b} to obtain 
\begin{align*}
  \int_{k-1}^k \int_k^{k+1} \frac{|f(x)-f(y)|^2}{|x-y|^2}\, dxdy 
    &\leqslant \int_{k-1}^k\int_{k-1}^k \frac{|f(x)-f(y)|^2}{|x-y|^2}\, dxdy \\
    &+ \int_k^{k+1}\int_k^{k+1} \frac{|f(x)-f(y)|^2}{|x-y|^2}\, dxdy \\
    &+ \int_0^t \frac{|f(k+t)-f(k-t)|^2}{t}\, dt
\end{align*}
and then sum over $k\in\dZ$. 
\end{proof}

\begin{example}\label{ex:4}
Let $G$ be the square graph $\sq_\delta$ with side length $\delta$. So 
$\sq_\delta$ has $4$ vertices that we will identify with the points $(0,0)$, 
$(\delta,0)$, $(\delta,\delta)$, $(0,\delta)$ in the plane and $4$ edges of 
length $\delta$. Then 
\begin{equation}\label{eq:2.10}
\begin{aligned}\
  \|f\|_{H^{1/2}(\sq_\delta)}^2 
    &= \int_0^\delta\int_0^\delta \frac{|f(x,0)-f(y,0)|^2}{|x-y|^2}\, dxdy + \int_0^\delta\int_0^\delta \frac{|f(\delta,x)-f(\delta,y)|^2}{|x-y|^2}\, dxdy \\
    &+ \int_0^\delta\int_0^\delta\frac{|f(x,\delta)-f(y,\delta)|^2}{|x-y|^2}\, dxdy + \int_0^\delta\int_0^\delta\frac{|f(0,x)-f(0,y)|^2}{|x-y|^2}\, dxdy \\
    &+ \int_0^\delta |f(x,0)-f(0,x)|^2\, \frac{dx}{x} + \int_0^\delta |f(x,0)-f(\delta,\delta-x)|^2\, \frac{dx}{x} \\
    &+ \int_0^\delta |f(x,\delta)-f(\delta,x)|^2\, \frac{dx}{x} + \int_0^\delta |f(0,x)-f(\delta-x,\delta)|^2\, \frac{dx}{x}\text{.}
\end{aligned}
\end{equation}
Although the $H^{1/2}(\sq_\delta)$ norm does not involve comparisons between 
values on opposite edges, it is not difficult to show bounds 
\begin{equation}\label{eq:2.11}
\begin{aligned}
  \int_0^1 |f(x,0)-f(x,\delta)|^2\, dx &\leqslant c\delta \|f\|_{H^{1/2}(\sq_\delta)}^2 \\
  \int_0^1 |f(0,y)-f(\delta,y)|^2\, dy &\leqslant c\delta \|f\|_{H^{1/2}(\sq_\delta)}^2\text{.}
\end{aligned}
\end{equation}
\end{example}

\begin{example}\label{ex:5}
Let $G$ be the graph paper graph $\gp_\delta$ with vertices at 
$\{(j\delta,k\delta)\}$, $j,k\in\dZ$ and horizontal and vertical edges of 
length $\delta$ joining $(j\delta,k\delta)$ with $((j+1)\delta,k\delta)$ and 
$(j\delta,k\delta)$ with $(j\delta,(k+1)\delta)$. The norm is given by 
\begin{equation}\label{eq:2.12}
\begin{aligned}
  \|f\|_{H^{1/2}(\gp_\delta)}^2 &= \sum_j \sum_k \int_0^\delta\int_0^\delta \frac{|f(j\delta+x,k\delta)-f(j\delta+y,k\delta)|^2}{|x-y|^2}\, dxdy \\
    &+ \sum_j\sum_k\int_0^\delta\int_0^\delta \frac{|f(j\delta,k\delta+x)-f(j\delta,k\delta+y)|^2}{|x-y|^2}\, dxdy \\
    &+ \sum_j\sum_k\int_{-\delta}^\delta |f(j\delta+x,k\delta)-f(j\delta,k\delta+x)|^2\, \frac{dx}{|x|} \\
    &+ \sum_j\sum_k \int_0^\delta |f(j\delta+x,k\delta)-f(j\delta-x,k\delta)|^2\, \frac{dx}{x} \\
    &+ \sum_j\sum_k \int_0^\delta |f(j\delta,k\delta+x)-f(j\delta,k\delta-x)|^2\, \frac{dx}{x}\text{.}
\end{aligned}
\end{equation}
Of course we could get an equivalent norm by deleting the last two sums in 
(\ref{eq:2.12}), as they are controlled by the third sum. We may regard 
$\gp_\delta$ as a countable union of square graphs $\sq_\delta$,and it is 
easily seen that $f\in H^{1/2}(\gp_\delta)$ if and only if the restriction of 
$f$ to each of the square graphs is in $H^{1/2}(\sq_\delta)$ with the sum of 
the squares of the norms $\|f\|_{H^{1/2}(\sq_\delta)}^2$ finite, and this 
gives an equivalent norm. 
\end{example}

\section{Traces of functions of finite energy}\label{sec:3}

Consider the homogeneous Sobolev space $H^1(\dR^2)$ of functions with finite 
energy 
\begin{equation}\label{eq:3.1}
  \|F\|_{H^1(\dR)}^2 = \int_{\dR^2} |\nabla F(x,y)|^2\, dxdy\text{.}
\end{equation}
These form a Hilbert space modulo constants. Functions of finite energy do not 
have to be continuous, as the example $F(x,y)=\log|\log(x^2+y^2)|$ (multiplied 
by an appropriate cutoff function) shows. However, it is well-known that these 
functions have well-defined traces on straight lines that are in 
$H^{1/2}(\dR)$, and $H^{1/2}(\dR)$ is the exact space of traces. Since the 
usual treatment of traces involves inhomogeneous Sobolev spaces we give the 
proof for the convenience of the reader. We omit the routine step of actually 
defining the traces and just prove the norm estimates. 

\begin{theorem}\label{thm:3.1}
The trace map $T:H^1(\dR)\to H^{1/2}(\dR)$ given formally by $TF(x)=F(x,0)$ is 
continuous, 
\begin{equation}\label{eq:3.2}
  \|T F\|_{H^{1/2}(\dR)} \leqslant c \|F\|_{H^1(\dR^2)}\text{.}
\end{equation}
Moreover there exists a continuous extension map 
$E:H^{1/2}(\dR) \to H^1(\dR^2)$ with $TEf=f$ and 
\begin{equation}\label{eq:3.3}
  \|E f\|_{H^1(\dR^2)} \leqslant c \|f\|_{H^{1/2}(\dR)}
\end{equation}
\end{theorem}
\begin{proof}
We work on the Fourier transform side, where 
\begin{align}\label{eq:3.4}
  \|F\|_{H^1(\dR^2)}^2 
    &= \int_{\dR^2} (\xi^2+\eta^2)|\hat F(\xi,\eta)|^2\, d\xi d\eta \qquad \text{and } \\ \label{eq:3.5}
  (T f)^\wedge(\xi) &= \int_{-\infty}^\infty \hat F(\xi,\eta)\, d\eta\text{.}
\end{align}
By Theorem \ref{thm:2.2a} we have 
\[
  \|T f\|_{H^{1/2}(\dR)}^2 = \int_{-\infty}^\infty \left| \int_{-\infty}^\infty \hat F(\xi,\eta)\, d\eta\right|^2\, |\xi|\, d\xi\text{.}
\]
By Cauchy-Schwarz we have 
\begin{align*}
  \left|\int_{-\infty}^\infty \hat F(\xi,\eta)\ d\eta\right|^2 
    &\leqslant \left(\int_{-\infty}^\infty (\xi^2+\eta^2) |\hat F(\xi,\eta)|^2\, d\eta\right)\left(\int_{-\infty}^\infty \frac{1}{\xi^2+\eta^2}\, d\eta\right) \\
    &= \frac{\pi}{|\xi|} \left(\int_{-\infty}^\infty (\xi^2+\eta^2) |\hat F(\xi,\eta)|^2\, d\eta\right) \qquad\text{so} \\
  \|T f\|_{H^{1/2}(\dR)}^2 
    &\leqslant \pi \int_{-\infty}^\infty \int_{-\infty}^\infty (\xi^2+\eta^2) |\hat F(\xi,\eta)|^2\, d\eta  \\
    &= \pi \|F\|_{H^1(\dR^2)}^2
\end{align*}
proving (\ref{eq:3.2}). 

Conversely, given $f\in H^{1/2}(\dR)$ define $E f=F$ by the Poisson integral 
\begin{equation}\label{eq:3.6}
  F(x,y) = \frac{|y|}{\pi} \int \frac{f(x-t)}{t^2+y^2}\, dt
\end{equation}
so that $T F=f$. Then 
\begin{equation}\label{eq:3.7}
  \hat F(\xi,\eta) = \frac{1}{\pi} \frac{\hat f(\xi) |\xi|}{\eta^2+|\xi|^2} \text{.}
\end{equation}
By (\ref{eq:3.4}) we have 
\begin{align*}
  \|F\|_{H^1(\dR^2)}^2 
    &= \frac{1}{\pi^2} \int_{\dR^2} \frac{|\hat f(\xi)|^2 |\xi|^2}{\eta^2+\xi^2}\, d\xi d\eta \\
    &= \frac{1}{\pi} \int_{-\infty}^\infty |\hat f(\xi)|^2 |\xi|\, d\xi 
\end{align*}
so we obtain (\ref{eq:3.3}) by Theorem \ref{thm:2.2a}. 
\end{proof}

Note that we define the extension $E f$ to be harmonic in each half-plane 
$y>0$ and $y<0$. Since harmonic functions minimize energy, our extension 
achieves the minimum $H^1(\dR^2)$ norm. 

There is a virtually identical trace theorem for functions of finite energy 
in the half-plane, say $y>0$ denoted $\dR_+^2$. To see this we only have to 
observe that an even reflection 
\begin{equation}\label{eq:3.8}
  R F(x,y) = F(x,-y) \qquad \text{for $y<0$}
\end{equation}
maps $H^1(\dR_+^2)$ continuously to $H^1(\dR^2)$. 

\begin{theorem}\label{thm:3.2}
The trace map $T:H^1(\dR_+^2)\to H^{1/2}(\dR)$ given formally by 
$T F(x)=F(x,0)$ is well-defined and bounded, and there exists a bounded 
extension map $E:H^{1/2}(\dR) \to H^1(\dR_+^2)$ with $T E f=f$, and the 
analogues of (\ref{eq:3.2}) and (\ref{eq:3.3}) hold. 
\end{theorem}

If we combine this with the well-known observation that energy is conformally 
invariant in the plane (not true in other dimensions, however), we obtain a 
powerful tool for obtaining trace theorems for other domains: find a conformal 
map between the domain and the half-space $\dR_+^2$, and transfer the 
$H^{1/2}(\dR)$ norm from the boundary of $\dR_+^2$ to the boundary of the 
domain, assuming the conformal map extends continuously to the boundary. 

A simple example is the strip $S=\{(x,y) : 0<y<\pi\}$. In complex notation 
$\varphi(z) = \log z$ is the conformal map from $\dR_+^2$ to $S$, with 
$\psi(z) = e^z$ its inverse. So $F\in H^1(S)$ if and only if 
$F\circ\varphi\in H^1(\dR_+^2)$ with equal norms. Then 
$f(t)=F(\varphi(t))\in H^{1/2}(\dR)$. Using Theorem \ref{thm:2.2a} this means 
\begin{equation}\label{eq:3.9}
\begin{aligned}
  \int_0^\infty \int_0^\infty \frac{|F(\log t)-F(\log s)|^2}{|t-s|^2}\, dtds 
    &+ \int_0^\infty\int_0^\infty \frac{|F(\log t+i\pi)-F(\log s+i\pi)|^2}{|t-s|^2}\, dtds \\
    &+ \int_0^\infty |F(\log t)-F(\log t+i\pi)|^2\, \frac{dt}{t} \\
    &\leqslant c \|F\|_{H^1(S)}^2 \text{.}
\end{aligned}
\end{equation}
The change of variable $x=\log t$, $y=\log s$ transforms the let hand side of 
(\ref{eq:3.9}) into 
\begin{equation}\label{eq:3.10}
\begin{aligned}
  \int_{-\infty}^\infty \int_{-\infty}^\infty |F(x)-F(y)|^2 \frac{e^x e^y}{|e^x-e^y|^2}\, dxdy 
    &+ \int_{-\infty}^\infty\int_{-\infty}^\infty |F(x+y)-F(x+i\pi)|^2 \frac{e^x e^y}{|e^x-e^y|}\, dxdy \\
    &+ \int_{-\infty}^\infty |F(x)-F(x+i\pi)|^2\, dx\text{.}
\end{aligned}
\end{equation}

To simplify the notation we split the trace of $F$ on the boundary of $S$ into 
two pieces, $T_0 F(x) = F(x)$ and $T_1 F(x) = F(x+i\pi)$, so that $T_0 F$ and 
$T_1 F$ are functions on $\dR$. 

\begin{theorem}\label{thm:3.3}
If $F\in H^1(S)$ then $T_0 F$ and $T_1 F$ are in $\tH^{1/2}(\dR)$ and 
$T_0 F-T_1 F\in L^2(\dR)$, with 
\begin{equation}\label{eq:3.11}
  \|T_0 F\|_{\tH^{1/2}(\dR)}^2 + \|T_1 F\|_{\tH^{1/2}(\dR)}^2 + \|T_0 F - T_1 F\|_2^2 \leqslant c \|F\|_{H^1(S)}\text{.}
\end{equation}
Conversely, given $f_0$ and $f_1$ in $\tH^{1/2}(\dR)$ with 
$f_0-f_1\in L^2(\dR)$, there exists $F=E(f_0,f_1)$ with $T_0 F=f_0$, 
$T_1 F = f_1$, $F\in H^1(S)$ with the reverse estimate of (\ref{eq:3.11}) 
holding. 
\end{theorem}
\begin{proof}
In view of (\ref{eq:3.10}) it suffices to show that 
\begin{equation}\label{eq:3.12}
  \int_{-\infty}^\infty\int_{-\infty}^\infty |F(x)-F(y)|^2 \frac{e^x e^y}{|e^xe^y|}{|e^x-e^y|^2}\, dxdy
\end{equation}
is bounded above and below by a constant multiple of 
\begin{equation}\label{eq:3.13}
  \iint_{|x-y|\leqslant 1} \frac{|F(x)-F(y)|^2}{|x-y|^2}\, dxdy = \|F\|_{\tH^{1/2}(\dR)}\text{.}
\end{equation}
Note that we may rewrite (\ref{eq:3.12}) as 
\begin{equation}\label{eq:3.14}
  \frac 1 4 \int_{-\infty}^\infty \int_{-\infty}^\infty \frac{|F(x)-F(y)|^2}{\left|\sinh\left(\frac{x-y}{2}\right)\right|^2}\, dxdy\text{.}
\end{equation}
It is clear that (\ref{eq:3.14}) is bounded below by a multiple of 
(\ref{eq:3.13}), and for the upper bound we need 
$\iint_{|x-y|\geqslant 1} \frac{|f(x)-f(y)|^2}{\left|\sinh\left((x-y)/2\right)\right|^2}$ 
bounded above by a multiple of (\ref{eq:3.13}), but this is a routine 
exercise because of the exponential decay of 
$\left|\sinh\left(\frac{x-y}{2}\right)\right|^{-2}$. 
\end{proof}

It might seem perplexing that the trace space on each of the lines is larger 
than $H^{1/2}(\dR)$, since in particular this implies that there are functions 
in $H^1(S)$ that do not extend to $H^1(\dR^2)$. However, it is easy to give an 
example of such a function: just take $F(x,y) = g(x)$ where $g(0) = 0$ for 
$x\leqslant 0$ and $g(x) = 1$ for $x\geqslant 1$ and $g$ is smooth in 
$[0,1]$. Then $\nabla F$ has compact support in $S$ so $F\in H^{1/2}(S)$, but 
$g\notin H^{1/2}(\dR)$. 

Another simple example is the first quadrant 
$Q=\{(x,y) : x>0\text{ and }y>0\}$. Then $\varphi(z) = \sqrt z$ is the 
conformal map of $\dR_+^2$ to $Q$, with inverse $\psi(z)=z^2$. Again it is 
convenient to split the trace into two parts mapping to functions on $\dR_+$, 
namely $T_0F(x) = F(x,0)$ and $T_1 F(x) = F(0,x)$. Since $F\in H^1(Q)$ if and 
only if $F\circ\varphi\in H^1(\dR_+^2)$, again by Theorem \ref{thm:2.2a} we 
have the expression 
\begin{equation}\label{eq:3.15}
\begin{aligned}
  \int_0^\infty\int_0^\infty\frac{|T_0 F(\sqrt t)-T_0 F(\sqrt s)|^2}{|t-s|^2}\, dsdt 
    &+ \int_0^\infty\int_0^\infty \frac{|T_1 F(\sqrt t)-T_1 F(\sqrt s)|^2}{|t-s|^2}\, dsdt \\
    &+ \int_0^\infty |T_0 F(\sqrt t)-T_1 F(\sqrt t)|^2\, \frac{dt}{t}
\end{aligned}
\end{equation}
for the trace norm. With the substitutions $t=x^2$, $s=y^2$ this becomes 
\begin{equation}\label{eq:3.16}
\begin{aligned}
  4 \int_0^\infty\int_0^\infty \frac{|T_0 F(x)-T_0 F(y)|^2}{|x-y|^2}\, \frac{xy}{|x+y|^2}\, dxdy 
    &+ 4\int_0^\infty\int_0^\infty\frac{|T_1 F(x)-T_1 F(y)|^2}{|x-y|^2}\, \frac{xy}{|x+y|^2}\, dxdy \\
    &+ 2\int_0^\infty |T_1 F(x)-T_1 F(x)|^2\, \frac{dx}{x}\text{.}
\end{aligned}
\end{equation}
It is easy to see that if $f_0,f_1\in H^{1/2}(\dR_+)$ and 
\begin{equation}\label{eq:3.17}
  \int_0^\infty |f_0(x)-f_1(x)|^2\, \frac{dx}{x} < \infty
\end{equation}
then there exists $F\in H^1(Q)$ with $T_0 F=f_0$ and $T_1 F=f_1$, because 
$\frac{xy}{|x+y|^2}$ is bounded. In other words, the function 
\[
  f(x) = \begin{cases}
           f_0(x) & \text{if $x>0$} \\
           f_1(x) & \text{if $x<0$}
         \end{cases}
\]
is in $H^{1/2}(\dR)$, and $\|F\|_{H^1(Q)} \leqslant c \|f\|_{H^{1/2}(\dR)}$. 
It is possible to show the converse statement as well, but this involves some 
technicalities since $\frac{xy}{|x+y|^2}$ is not bounded below. It is easier 
to observe that $F\in H^1(Q)$ may be extended by even reflection across the 
axes to a function in $H^1(\dR^2)$, so the even reflections of $T_0 F$ and 
$T_1 F$ must be in $H^{1/2}(\dR^2)$, so $T_0 F$ and $T_1 F$ must be in 
$H^{1/2}(\dR_+)$, and we already have (\ref{eq:3.17}) for $f_0=T_0 F$, 
$f_1=T_1 F$. A direct proof of (\ref{eq:3.17}) is possible but involves 
technicalities. 

Another simple example is the unit disk $D$, with 
$\varphi(z) = \frac{1-z}{1+z}$ the conformal mapping of $\dR_+^2$ to $D$. 
The trace space of $H^1(D)$ is $H^{1/2}(C)$ for $C$ the unit circle with norm 
\begin{equation}\label{eq:3.18}
  \|f\|_{H^{1/2}(C)}^2 = \int_0^{2\pi}\int_0^{2\pi} \frac{|f(e^{i\theta})-f(e^{i\theta'})|^2}{4\left|\sin \frac 1 2(\theta-\theta')\right|^2}\, d\theta d\theta' \text{.}
\end{equation}
Of course $2\left|\sin \frac 1 2(\theta-\theta')\right|$ is exactly the 
chordal distance $|e^{i\theta}-e^{i\theta'}|$. It is interesting to observe 
that exactly the same trace space arises from the exterior of the circle 
$\{|z|>1\}$, as $z\mapsto 1/\bar z$ is an anticonformal map of $D$ to this 
exterior domain that agrees with $\varphi(z)$ on the circle. Similarly, for a 
circle $C_r$ of radius $r$, the analog of (\ref{eq:3.18}) is 
\begin{equation}\label{eq:3.19}
  \|f\|_{H^{1/2}(C_r)}^2 = \int_0^{2\pi}\int_0^{2\pi} \frac{|f(re^{i\theta}-f(re^{i\theta'})|^2}{4\left|r\sin \frac 1 2(\theta-\theta')\right|^2}\, rd\theta r d\theta'\text{.}
\end{equation}

Of course it is not necessary to use a conformal map $\varphi$. A Lipschitz 
map or even a quasiconformal map changes the $H^1$ norm by a bounded amount. 
So for $\sq_\delta^\circ$, the interior of the square $\sq_\delta$, the trace 
space of $H^1(\sq_\delta^\circ)$ is $H^{1/2}(\sq_\delta)$ with norm given by 
(\ref{eq:2.10}), since one can ``square the circle'' with a Lipschitz map. 

Next we consider traces on infinite collections of lines. First consider the 
horizontal line collection $\hlc=\{(x,n\pi) : x\in \dR,n\in\dZ\}$. For a 
function $F$ in $H^1(\dR^2)$ define the traces $T_n F(x)=F(x,\pi n)$. 

\begin{theorem}\label{thm:3.4}
A set of functions $\{f_n\}$ on $\dR$ are the traces $f_n=T_n F$ for 
$F\in H^1(\dR^2)$ if and only if $f_n\in \tH^{1/2}(\dR)$ and 
$f_n-f_{n+1}\in L^2(\dR)$ with 
\begin{equation}\label{eq:3.20}
  \sum_n \|f_n\|_{\tH^{1/2}(\dR)}^2 + \sum_n \|f_n-f_{n+1}\|_{L^2(\dR)}^2 < \infty\text{,}
\end{equation}
and the corresponding norm equivalence holds.
\end{theorem}
\begin{proof}
Basically we just have to apply Theorem \ref{thm:3.3} to each of the strips 
$\{n\pi < y < (n+1)\pi\}$ and sum (\ref{eq:3.11}) over all the strips. To do 
this we just have to observe that a function belongs to $H^1(\dR^2)$ if and 
only if its restriction to each strip is in $H^1$ of that strip, the traces 
agree on neighboring strips, and the sum of the energies is finite. 
\end{proof}

There is something a bit unsettling about this result. We know that 
$f_n=T_n F$ actually belongs to the smaller space $H^{1/2}(\dR)$ for 
$F\in H^1(\dR^2)$, yet this space plays no role in the characterization 
(\ref{eq:3.20}). It is an indirect consequence of the theorem that if 
$\{f_n\}$ is a family of functions satisfying (\ref{eq:3.20}), then each $f_n$ 
is indeed in $H^{1/2}(\dR)$. It should be possible to prove this directly, but 
again this seems rather technical. Note that we only get a uniform bound for 
$\|f_n\|_{H^{1/2}(\dR)}^2$. The following example shows that we can't do too 
much better than this (most likely $\|f_n\|_{H^{1/2}(\dR)}^2 = o(1)$). 

Consider the function $F(x,y)=(1+x^2+y^2)^{-\alpha}$ for $\alpha>0$. A direct 
computation shows that 
$|\nabla F(x,y)|\leqslant 2\alpha (1+x^2+x^y)^{-\alpha-\frac 1 2}$, so 
$F\in H^1(\dR^2)$. Now 
\[
  T_n F(x) = (1+\pi^2 n^2 + x^2)^{-\alpha} 
           = (1+\pi^2 n^2)^{-\alpha} g\left(\frac{x}{\sqrt{1+\pi^2 n^2}}\right)
\]
for $g(x) = (1+x^2)^{-\alpha}$. It is easy to see that $g\in H^{1/2}(\dR)$, 
so by dilation invariance of the $H^{1/2}(\dR)$ norm we see that 
$\|T_n F\|_{H^{1/2}(\dR)}^2 = c(1+\pi^2 n^2)^{-2\alpha}$ so 
$\sum \|T_n F\|_{H^{1/2}(\dR)}^2 = \infty$ for $\alpha\leqslant \frac 1 4$. 

Next we consider the trace on the graph paper graph $\gp_\delta$. 

\begin{theorem}\label{thm:3.5}
The trace space of $H^1(\dR^2)$ on $\gp_\delta$ is exactly $H^1(\gp_\delta)$ 
with norm given by (\ref{eq:2.12}). 
\end{theorem}
\begin{proof}
We simply use the trace theorem of $H^1$ on each $\delta$-square that makes 
up $\gp_\delta$ and add.
\end{proof}

In place of square graph paper we could consider triangular graph paper 
$\tgp_\delta$ consisting of the tiling of the plane by equilateral triangles 
of side length $\delta$. Then the analog of Theorem \ref{thm:3.5} holds with 
essentially the same proof.

\section{The graph paper trace characterization}\label{sec:4}

In this section we fix an integer $m\geqslant 2$, and consider the sequence of 
graph paper graphs $\gp_{m^n}$, thought of as the unions of the edges, or 
equivalently the countable union of horizontal and vertical lines in the 
plane with $m^n$ separation. These are nested subsets of the plane, 
$\gp_{m^n} \subset \gp_{m^{n'}}$ if $n'<n$ and we are interested in the limit 
as $n\to -\infty$, so the graph paper gets increasingly finer.

We let $T_n$ denote the trace map from functions defined on $\dR^2$ to 
$\gp_{m^n}$. By the nesting property we may also consider $T_n$ to be defined 
on functions on $\gp_{m^{n'}}$ with $n'<n$. Our goal is to characterize 
functions in $H^1(\dR^2)$ by their traces $T_n F$. 

\begin{theorem}\label{thm:4.1}
a) Let $F\in H^1(\dR^2)$. Then $T_n F\in H^{1/2}(\gp_{m^n})$ for all $n$ with 
uniformly bounded norms, and 
\begin{equation}\label{eq:4.1}
  \sup_{n\in\dZ} \|T_n F\|_{H^{1/2}(\gp_{m^n})}^2 \leqslant c \|F\|_{H^1(\dR)}^2
\end{equation}

b) Let $f_n\in H^{1/2}(\gp_{m^n})$ be a sequence of functions with uniformly 
bounded norms satisfying the consistency condition $T_n f_{n'} = f_n$ if 
$n'<n$. Then there exists $F\in H^1(\dR^2)$ such that $T_n F=f_n$ and 
\begin{equation}\label{eq:4.2}
  \|F\|_{H^1(\dR^2)}^2 \leqslant c \sup_{n\in\dZ} \|f_n\|_{H^{1/2}(\gp_{m^n})}^2
\end{equation}
\end{theorem}
\begin{proof}
Part a) is an immediate consequence of Theorem \ref{thm:3.5}. To prove b) we 
define $F_n$ to be the harmonic extension of $f_n$ into each of the graph 
paper spaces. Since these harmonic extensions minimize energy, we have 
$F_n\in H^1(\dR^2)$ and 
\[
  \|F_n\|_{H^1(\dR^2)} \leqslant c \|f_n\|_{H^{1/2}(\gp_{m^n})}\text{,}
\]
again by Theorem \ref{thm:3.5}. Thus there exists a subsequence 
$n_j\to -\infty$ such that $F_{n_j}$ converges in the weak topology of 
$H^1(\dR^2)$ to a function $F$ satisfying (\ref{eq:4.2}). It remains to show 
that the weak convergence respects traces, so that $T_n F_{n_j} = f_n$ for all 
$n_j$ implies $T_n F = f_n$. 

But the equality of traces on $\gp_{m^n}$ is the same as equality of traces on 
each of the lines that make up $\gp_{m^n}$; and since all lines are 
essentially equivalent, it suffices to show that $F_{n_j}(x,0)$ converges 
weakly in $H^{1/2}(\dR)$ to $F(x,0)$. This is most easily seen on the Fourier 
transform side, where both $H^1(\dR^2)$ and $H^{1/2}(\dR)$ are just weighted 
$L^2$ spaces. 

The weak convergence $F_{n_j}\to F$ in $H^1(\dR^2)$ says 
\begin{equation}\label{eq:4.3}
  \iint \hat F_{n_j}(\xi,\eta) G(\xi,\eta)(\xi^2+\eta^2)\, d\xi d\eta \to \iint \hat F(\xi,\eta) G(\xi,\eta)(\xi^2+\eta^2)\, d\xi d\eta
\end{equation}
for every $G\in L^2\left((\xi^2+\eta^2)d\xi d\eta\right)$. The weak 
convergence $F_{n_j}(x,0)\to F(x,0)$ requires that we show 
\begin{equation}\label{eq:4.4}
  \int\left(\int \hat F_{n_j}(\xi,\eta)\, d\eta\right) H(\xi)|\xi|\, d\xi \to \int\left(\int \hat F(\xi,\eta)\, d\eta\right) H(\xi) |\xi|\, d\xi
\end{equation}
for every $H\in L^2\left(|\xi|\, d\xi\right)$. So given $H$, choose 
\begin{equation}\label{eq:4.5}
  G(\xi,\eta) = \frac{|\xi| H(\xi)}{\xi^2+\eta^2} \text{.}
\end{equation}
Since 
\begin{align*}
  \iint |G(\xi,\eta)|^2(\xi^2+\eta^2)\, d\xi d\eta 
    &= \int\left(\int \frac{|\xi|^2}{\xi^2+\eta^2}\, d\eta\right)|H(\xi)|^2\, d\xi \\
    &= \pi \int |H(\xi)|^2 |\xi|\, d\xi
\end{align*}
we may use the choice of $G$ in (\ref{eq:4.2}). But then (\ref{eq:4.3}) and 
(\ref{eq:4.4}) are identical. 
\end{proof}

This result localizes in several ways. For example, if $F\in H^1(\dR^2)$ and 
we wish to estimate the amount of energy that is contained in an open set 
$\Omega$, that is 
\begin{equation}\label{eq:4.6}
  \int_\Omega |\nabla F|^2\, dxdy\text{,}
\end{equation}
we just have to take the sum of the terms in (\ref{eq:2.12}) that correspond 
to edges contained in $\Omega$. Denote this sum by 
$\|T_n F\|_{H^{1/2}(\Omega\cap \gp_{m^n})}^2$. Then (\ref{eq:4.6}) is bounded 
above and below by a constant times 
\begin{equation}\label{eq:4.7}
  \sup_{n\in\dZ} \|T_n F\|_{H^{1/2}(\Omega\cap\gp_{m^n})} \text{.}
\end{equation}
We obtain the same norm equivalence if we only assume $F\in H^1(\Omega)$, 
meaning (\ref{eq:4.6}) is finite. (Note that this does not say anything about 
the trace of $F$ on the boundary of $\Omega$.) Also, we may start by assuming 
that $F\in H_\text{loc}^1(\dR^2)$, meaning that (\ref{eq:4.6}) is finite 
whenever $\Omega$ is bounded, and obtain the norm equivalence of 
(\ref{eq:4.6}) and (\ref{eq:4.7}). 

The same result will also hold if we replace $\gp_{m^n}$ by the triangular 
$\tgp_{m^n}$. 

It is clear that we may replace the $\sup$ in (\ref{eq:4.2}) and 
(\ref{eq:4.7}) by the $\limsup$ as $n\to-\infty$. It is not clear that a limit 
has to exist, however, since we only have estimates above and below, rather 
than identity, for our norms. 

We can also characterize functions of finite energy by their traces on pencils 
of parallel lines of equal separation; in other words, the horizontal lines in 
$\gp_{m^n}$. Denote this by $\pp_{m^n}$. We will use Theorem \ref{thm:3.4}, 
but the norms defined by (\ref{eq:3.20}) are not dilation invariant. That 
means we want to define $\tH^{1/2}(\pp_{m^n})$ by the finiteness of 
\begin{equation}\label{eq:4.8}
\begin{aligned}
  \sum_{k\in\dZ} \iint_{|x-y|\leqslant m^n} &\frac{|f(x,k m^n) - f(y,k m^n)|^2}{|x-y|^2}\, dxdy \\
    &+ \sum_{k\in\dZ} m^{-n} \int_{-\infty}^\infty |f(x,(k+1) m^n)-f(x,k m^n)|^2\, dx\text{,}
\end{aligned}
\end{equation}
and we define this to be $\|f\|_{\tH^{1/2}(\pp_{m^n})}^2$. Then the analog of 
Theorem \ref{thm:4.1} holds with $T_n F$ equal to the trace on $\pp_{m^n}$ and 
$H^{1/2}(\gp_{m^n})$ replaced by $\tH^{1/2}(\pp_{m^n})$. The proof is 
essentially the same, using the scaled version of Theorem \ref{thm:3.4} with 
(\ref{eq:4.8}) in place of (\ref{eq:3.20}).

\section{Fractals}\label{sec:5}

The Sierpinski gasket ($\sg$) is the self-similar fractal defined by the 
identity 
\begin{equation}\label{eq:5.1}
  \sg = \bigcup_{i=0}^2 \Phi_i(\sg)
\end{equation}
where $\Phi_i$ are the homothety maps of the plane 
$\Phi_i(x) = \frac 1 2 x + \frac 1 2 q_i$ and $\{q_0,q_1,q_2\}$ are the 
vertices of an equilateral triangle with side length $1$. $\sg$ is the unique 
nonempty compact subset of the plane satisfying (\ref{eq:5.1}). The mappings 
$\{\Phi_i\}$ comprise what is called an \emph{iterated function system}, and 
the iterates of the mappings are denoted 
$\Phi_w = \Phi_{w_1}\circ\cdots\circ\Phi_{w_m}$ where $w=(w_1,\dotsc,w_m)$ is 
a word of length $|w|=m$ and each $w_j=0$, $1$, or $2$. Then by iterating 
(\ref{eq:5.1}) we obtain 
\begin{equation}\label{eq:5.2}
  \sg=\bigcup_{|w|=m} \Phi_w(\sg)
\end{equation}
expressing $\sg$ as a union of $3^m$ miniature gaskets (called 
\emph{$m$-cells}) that are similar to $\sg$ with similarity ratio $2^{-m}$. 
Note that $\sg$ has the \emph{post-critically finite} (PCF) property that 
distinct $m$-cells can intersect only at the vertices $\Phi_w q_i$. For this 
reason we refer to $\{q_i\}$ as the \emph{boundary} of $\sg$, and 
$\{\Phi_w q_i\}$ as the boundary of the $m$-cell $\Phi_m(\sg)$, although these 
are not boundaries in the topological sense. 

We may approximate $\sg$ by the metric graphs $\sg_m=\sg\cap \tg_{2^{-m}}$. So 
the vertices are $\{\Phi_w q_i\}$, for $|w|=m$ and $i=0,1,2$, the edges are 
$\{\Phi_w e_{i j}\}$ for $|w|=m$ and $e_{i j}$ is the edge of the original 
triangle joining $q_i$ and $q_j$, and $\Phi_w e_{ij}$ has length $2^{-m}$. Let 
\begin{equation}\label{eq:5.3}
  E_m(f) = \sum_{i\ne j} \sum_{|w|=m} |f(\Phi_w q_i) - f(\Phi_w q_j)|^2
\end{equation}
denote the unrenormalized graph energy on $\sg_m$. Kigami (see \cite{8,12}) 
defines an energy on $\sg$ by 
\begin{equation}\label{eq:5.4}
  \cE(f) = \lim_{m\to \infty} \left(\frac 5 3\right)^m E_m (f)\text{.}
\end{equation}
The renormalization factor $(5/3)^m$ may be explained as follows: the 
sequence $(5/3)^m E_m(f)$ is always nondecreasing, and there exists a 
$3$-dimensional space of harmonic functions for which it is constant. We can 
then define $\dom \cE$, the space of functions of finite energy, as those 
functions for which (\ref{eq:5.4}) is finite. This is a space of continuous 
functions on $\sg$ that forms an infinite dimensional Hilbert space (after 
modding out by the constants) with norm $\cE(f)^{1/2}$. This energy satisfies 
the self-similar identity 
\begin{equation}\label{eq:5.5}
  \cE(f) = \sum_{i=0}^2 \left(\frac 5 3\right) \cE(f\circ \Phi_i)
\end{equation}
and satisfies the axioms for a local regular Dirichlet form (\cite{4}). Up to 
a constant multiple it is the only Dirichlet form with those properties. It is 
also symmetric with respect to the $D_3$ symmetry group of the triangle. This 
energy forms the basic building block for a whole theory of analysis on $\sg$, 
including a theory of Laplacians. We will not be using this wider theory here, 
but direct the curious reader to \cite{8,12} for details. 

Since the functions in $\dom\cE$ are continuous, there is no problem defining 
traces $T_m$ on $\sg_m$. The problem of characterizing the trace space 
$T_n(\sg)$ on the boundary of the triangle has been solved by Jonsson 
\cite{6,7} (see \cite{5} for a different proof) in terms of Sobolev spaces of 
order $\beta$, with $\beta=\frac 1 2+\frac{\log 5/3}{\log 4}$. Note that 
$\frac 1 2<\beta<1$. For any metric graph $G$ we define $H^\beta(G)$ (for any 
$\beta$ in the above range) to be the space of continuous functions such that 
\begin{equation}\label{eq:5.6}
  \|F\|_{H^\beta(G)}^2 = \sum_{e\in E}\int_0^{L_e}\int_0^{L_e} \frac{|F(e(x))-F(e(y))|^2}{|x-y|^{1+2\beta}}\, dxdy 
\end{equation}
is finite. Note that in contrast to (\ref{eq:2.1}), there is no term comparing 
values on intersecting edges, since the continuity condition takes care of the 
comparison (this idea is also used in \cite{11}). We then have the following 
result analogous to Theorem \ref{thm:3.1}. 

\begin{proposition}[\cite{5,6,7}]\label{prop:5.1}
The trace map $T_0$ is continuous from $\dom\cE$ to $H^\beta(\sg_0)$ with 
$\beta=\frac 1 2+\frac{\log 5/3}{\log 4}$ with 
\begin{equation}\label{eq:5.7}
  \|T_0 F\|_{H^\beta(\sg_0}^2 \leqslant c \cE(F)\text{.}
\end{equation}
Moreover, there exists a continuous linear extension map 
$E_0:H^\beta(\sg_0) \to \dom\cE$ with $T_0 E_0 f = f$ and 
\begin{equation}\label{eq:5.8}
  \cE(E_0 f) \leqslant c \|f\|_{H^\beta(\sg_0)}^2\text{.}
\end{equation}
\end{proposition}

We note that \cite{5,6,7} use a slightly different, but equivalent norm for 
$H^\beta(\sg_0)$. 

Next we need to obtain the analogous statement for the trace map $T_m$ to 
$\sg_m$. We note that energy is additive for continuous functions, and in view 
of the self-similarity (\ref{eq:5.5}) iterated, 
\begin{equation}\label{eq:5.9}
  \cE(F)=\sum_{|w|=m} \left(\frac 5 3\right)^m \cE(F\circ \Phi_w)\text{,}
\end{equation}
and if we apply (\ref{eq:5.8}) to $F\circ\Phi_w$ we have 
\begin{equation}\label{eq:5.10}
  \sum_{|w|=m}\left(\frac 5 3\right)^m \|T_0 F\circ \Phi_w\|_{H^\beta(\sg_0)}^2 \leqslant c \sum_{|w|=m} \left(\frac 5 3\right)^m \cE(F\circ \Phi_w) = c \cE(F)
\end{equation}
by (\ref{eq:5.9}). Now we observe that $\sg_m=\bigcup_{|w|=m} \Phi_w(\sg_0)$, 
and this is a disjoint union of edges, since each edge is just a side of a 
triangle $\Phi_w(\sg_0)$ for some $w$ with $|w|=m$. 

So consider one of these edges, $\Phi_w(e_{i j})$. It is parameterized by $x$ 
in the interval $[0,2^{-m}]$, and the contribution (\ref{eq:5.6}) is 
\begin{equation}\label{eq:5.11}
\begin{aligned}
  \int_0^{2^{-m}}\int_0^{2^{-m}} &\frac{|F(e(x))-F(e(y))|^2}{|x-y|^{1+2\beta}}\, dxdy \\
    &= \frac{4^m}{2^{1+2\beta}} \int_0^1\int_0^1 \frac{|F(\Phi_w(e_{i j}(x))) - F(\Phi_w(e_{i j}(y)))|^2}{|x-y|^{1+2\beta}}\, dxdy
\end{aligned}
\end{equation}
after a change of variables. Summing all the contributions over all the edges 
in $\sg_m$ yields 
\begin{equation}\label{eq:5.12}
  \|T_m F\|_{H^\beta(\sg_m)}^2 = \sum_{|w|=m} \frac{4^m}{2^{(1+2\beta)m}} \|T_0 F\circ \Phi_w\|_{H^\beta(\sg_0)}^2
\end{equation}
by (\ref{eq:5.11}). But the choice of $\beta$ makes 
$\frac{4}{2^{1+2\beta}} = \frac 5 3$, so (\ref{eq:5.12}) combined with 
(\ref{eq:5.10}) yields 
\begin{equation}\label{eq:5.13}
  \|T_m F\|_{H^\beta(\sg_m)}^2 \leqslant c \cE(F)\text{.}
\end{equation}
This is the exact analog of (\ref{eq:5.7}). 

\begin{theorem}\label{thm:5.2}
The trace map $T_m$ is continuous from $\dom \cE$ to $H^\beta(\sg_m)$ for 
$\beta$ as in Proposition \ref{prop:5.1} and the estimate (\ref{eq:5.13}) 
holds. Moreover, there exists a continuous linear extension map 
$E_m:H^\beta(\sg_m)\to \dom\cE$ with $T_m E_m f = f$ and 
\begin{equation}\label{eq:5.14}
  \cE(E_m f) \leqslant c \|f\|_{H^\beta(\sg_m)}^2 \text{.}
\end{equation}
\end{theorem}
\begin{proof}
We have already established (\ref{eq:5.13}). To define the extension map $E_m$ 
we set 
\begin{equation}\label{eq:5.15}
  E_m(f) = \Phi_w^{-1} E_0(f\circ \Phi_w) \qquad \text{on $\Phi_w(\sg)$.}
\end{equation}
Note that $E_m(f)$ is continuous, because at the boundary points of the 
$m$-cells that make up $\sg_m$ we have $E_m(f) = f$. The same reasoning that 
obtains (\ref{eq:5.13}) from (\ref{eq:5.7}) also leads from (\ref{eq:5.8}) to 
(\ref{eq:5.14}). 
\end{proof}

Next we have the analog of Theorem \ref{thm:4.1}. 

\begin{theorem}\label{thm:5.3}
a) Let $F\in \dom\cE$. Then $T_m F\in H^\beta(\sg_m)$ for all $m$ with 
uniformly bounded norms, and 
\begin{equation}\label{eq:5.16}
  \sup_m \|T_m F\|_{H^\beta(\sg_m)}^2 \leqslant c \cE(F)
\end{equation}

b) Let $f_m\in H^\beta(\sg_m)$ be a sequence of functions with 
uniformly bounded norms satisfying the consistency condition 
$T_m f_{m'} = f_m$ if $m\leqslant m'$. Then there exists $F\in\dom\cE$ such 
that $T_m f=f_m$ and 
\begin{equation}\label{eq:5.17}
  \cE(F)\leqslant c\sup_m \|f_m\|_{H^\beta(\sg_m)}^2 \text{.}
\end{equation}
\end{theorem}
\begin{proof}
(\ref{eq:5.16}) is an immediate ff consequence of (\ref{eq:5.13}). To prove 
b) construct a sequence of functions $F_m$ by taking the harmonic (energy 
minimizing) extension of $f_m$ from $\sg_m$ to $\sg$. Then by (\ref{eq:5.14}), 
the sequence $\{F_m\}$ is uniformly bounded in $\dom\cE$. A quantitative 
version of the continuity of functions in $\dom\cE$ implies that the sequence 
$\{F_m\}$ is also uniformly equicontinuous. Thus by passing to a subsequence 
twice we can find a subsequence $\{F_{m_j}\}$ that converges both weakly in 
the Hilbert space $\dom\cE$ and uniformly to a function $F$ in $\dom\cE$ with 
the estimate (\ref{eq:5.17}) holding. Because the convergence is pointwise and 
the consistency condition holds we have $T_{m_j} F = F_{m_j} = f_{m_j}$ on 
$\sg_{m_j}$, so $T_m F = f_m$. 
\end{proof}

The second example of a fractal we consider is the Sierpinski carpet 
($\SC$), again defined by a self-similar identity 
\begin{equation}\label{eq:5.18}
  \SC = \bigcup_{i=1}^8 \Phi_i(\SC)
\end{equation}
where now $\Phi_i$ are the homothety maps of the plane with contraction ratio 
$1/3$ mapping the unit square into $8$ of the $9$ subsquares of side length 
$1/3$ (all except the central subsquare). This self-similar fractal is not 
PCF, so the method of Kigami cannot be used to construct an energy. 
Nevertheless, two approaches due to Barlow and Bass and Kusuoka and Zhou 
\cite{1} were given in the late 1980's, and recently in \cite{2} it was shown 
that up to a constant multiple there is a unique self-similar energy, so both 
approaches yield the same energy. Once again, all functions in $\dom\cE$ are 
continuous. The self-similar identity for the energy here is 
\begin{equation}\label{eq:5.19}
  \cE(F) = \sum_{i=1}^8 r \cE(F\circ \Phi_i)\text{,}
\end{equation}
where $r$ is a constant whose exact value has not been determined ($r$ is 
slightly larger than $1.25$). 

Again we may approximate $\SC$ by a sequence of metric graphs, $\{\SC_m\}$, 
with $\SC_m = \SC\cap \gp_{3^{-m}}$. Thus, the edges of $\SC_m$ have length 
$3^{-m}$ and are of the form $\Phi_w(e_i)$ with $|w|=m$, where 
$e_1,e_2,e_3,e_4$ are the boundary edges of the unit square. Again let $T_m$ 
denote the trace map onto $\SC_m$. The trace space for $T_0$ has been 
identified by Hino and Kumagai \cite{5} as the Sobolev space $H^\beta(\SC_0)$ 
with $\beta=\frac 1 2+\frac{\log r}{\log 9}$. Note that again 
$\frac 1 2<\beta<1$. 

\begin{proposition}[\cite{5}]\label{prop:5.4}
The trace map $T_0$ is continuous from $\dom\cE$ to $H^\beta(\SC_0)$ for 
$\beta=\frac 1 2+\frac{\log r}{\log 9}$ with 
\begin{equation}\label{eq:5.20}
  \|T_0 F\|_{H^\beta(\SC_0)}^2 \leqslant c \cE(F)\text{.}
\end{equation}
Moreover, there exists a continuous linear extension map 
$E_0:H^\beta(\SC_0)\to \dom\cE$  with $T_0 E_0 f = f$ and 
\begin{equation}\label{eq:5.21}
  \cE(E_0 f) \leqslant c \|f\|_{H^\beta(\SC_0)}^2 \text{..}
\end{equation}
\end{proposition}

We now claim that the analogs of Theorem \ref{thm:5.2} and \ref{thm:5.3} hold 
for $\SC$ in place of $\sg$, with essentially the same proof. The only detail 
that needs to be checked is the dilation argument. In this case the 
contribution to (\ref{eq:5.6}) from the edge $e=F_w(e_1)$ is 
\begin{equation}\label{eq:5.22}
\begin{aligned}
  \int_0^{3^{-m}}\int_0^{3^{-m}} &\frac{|F(e_i(x))-F(e_i(y))|^2}{|x-y|^{1+2\beta}}\, dxdy \\
    &= \frac{9^m}{3^{1+2\beta}} \int_0^1\int_0^1 \frac{|F(\Phi_w(e_i(x)))-F(\Phi_w(e_i(y)))|^2}{|x-y|^{1+2\beta}}\, dxdy
\end{aligned}
\end{equation}
after a change of variable, as the analog of (\ref{eq:5.11}). We note that 
$\frac{9}{3^{1+2\beta}}=r$ in this case, so summing (\ref{eq:5.22}) yields the 
analog of (\ref{eq:5.13}) as a subsequence of (\ref{eq:5.20}). The rest of the 
arguments are the same. 

For our final fractal example we consider the classical Julia sets of complex 
polynomials. Fix a polynomial $P(z)$ (of degree at least two) and let 
$\cJ$ denote its Julia set. We assume $\cJ$ is connected. In many cases (see 
\cite{13}) it is possible to parameterize $\cJ$ by the unit circle as follows. 
Let $\Omega$ denote the unbounded component of the complement of $\cJ$ in 
$\dC$, so $\Omega\cup\{\infty\}$ is simply connected, and let $\varphi$ be a 
conformal map from $\{z:|z|>1\}$ to $\Omega$. In many cases $\varphi$ extends 
continuously to the boundary circle, and this maps $C$ onto $\cJ$ (usually 
not one-to-one). Although there is usually no useful formula for $\varphi$, in 
many cases it is possible to describe explicitly the points on $C$ that are 
identified under $\varphi$. There have been a number of papers that utilize 
this parametrization to construct an energy on $\cJ$ \cite{14,15,16,17}. 

Here we deal with a different question: how to characterize the traces on 
$\cJ$ of functions of finite energy on $\Omega$. The answer is almost 
immediate using the methods of section \ref{sec:3}. We know that 
$F\in H^1(\Omega)$ if and only if $F\circ\varphi\in H^1(|z|>1)$, and the space 
of traces of $F\circ\varphi$ on $C$ is exactly $H^{1/2}(C)$. Thus the space of 
traces of $F$ on $\cJ$, that we should denote $H^{1/2}(\cJ)$, is characterized 
by the finiteness of 
\begin{equation}\label{eq:5.23}
  \|F\|_{H^{1/2}(\cJ)}^2 = \int_0^{2\pi}\int_0^{2\pi} \frac{|F(\varphi(e^{i\theta}))-F(\varphi(e^{i\theta'}))|^2}{4\sin^2\frac 1 2(\theta-\theta')}\, d\theta d\theta'\text{.}
\end{equation}
One could perhaps hope for a more direct characterization in terms of an 
integral involving $|F(z)-F(z')|^2$ as $z$ and $z'$ vary over $\cJ$. This 
would involve choosing a measure on $\cJ$ (there are more than one natural 
choices) and finding the appropriate denominator in terms of a distance from 
$z$ to $z'$ on $\cJ$. Good luck!

\end{document}